\documentclass[11p,reqno]{amsart}
\usepackage{amssymb,mathrsfs,graphicx}
\usepackage{nccmath}
\usepackage{ifthen}
\usepackage{colortbl}
\usepackage{color}
\usepackage{hyperref}
\usepackage{url}
\definecolor{black}{rgb}{0.0, 0.0, 0.0}
\definecolor{red}{rgb}{1.0, 0.5, 0.5}
\provideboolean{shownotes} 
\setboolean{shownotes}{false} 
\newcommand{\margnote}[1]{
	\ifthenelse{\boolean{shownotes}}%
	{\marginpar{\raggedright\tiny\texttt{#1}}}%
	{}%
}
\newcommand{\hole}[1]{
	\ifthenelse{\boolean{shownotes}}%
	{\begin{center} \fbox{ \rule {.25cm}{0cm} \rule[-.1cm]{0cm}{.4cm}
				\parbox{.85\textwidth}{\begin{center} \texttt{#1}\end{center}} \rule
				{.25cm}{0cm}}\end{center}} {} }

\title[]{On the eigenvalues of a class of matrices with displacement structure arising in optimal control}

\author[Peters]{Andr\'es A. Peters}
\address[Andr\'es A. Peters]{\newline Advanced Center for Electrical and Electronic Engineering \newline
	Universidad T\'ecnica Federico Santa Mar\'ia, Valparaíso, Chile}
\email{andres.peters@usm.cl}

\author[Vargas]{Francisco J. Vargas}
\address[Francisco J. Vargas]{\newline Advanced Center for Electrical and Electronic Engineering \newline
	Universidad T\'ecnica Federico Santa Mar\'ia, Valparaíso, Chile}
\email{francisco.vargasp@usm.cl}

{

	\numberwithin{equation}{section}
	
	\newtheorem{theorem}{Theorem}[section]
	
	\newtheorem{corollary}{Corollary}[section]
	\newtheorem{proposition}{Proposition}[section]
	\newtheorem{remark}{Remark}[section]

\def\({\begin{eqnarray}}
\def\){\end{eqnarray}}
\def\[{\begin{eqnarray*}}
\def\]{\end{eqnarray*}}

\begin{document}
\allowdisplaybreaks

\date{\today}
\subjclass[]{}
\keywords{}
	
	\begin{abstract}
		In this work we present a framework for studying the eigenvalues of a family of matrices with a particular displacement  structure. The family admits a specific decomposition as the product of an upper and a lower triangular matrices having an increasing number of real parameters in predefined positions. Similar matrices appear naturally when solving some kinds of optimal control problems. In our case, as stated by Nehari's theorem, the eigenvalues and eigenvectors fully characterize the solution. Commonly, such problems are solved by numerical means, making it difficult to obtain insight in the role that the parameters play on the solution. Our results provide a framework that enables to compute individually, under some simple assumptions, the eigenvalues of the matrices as roots of a monotone transcendental function with many desirable properties. In order to do so, we first obtain a three-term recursive characterization of the corresponding characteristic polynomials. This enables the aforementioned representation. Our framework also allows for the computation of bounds, numerical methods and even analytical characterizations with closed form solutions, whenever the problem parameters satisfy simple conditions.
	\end{abstract}
	
\maketitle \centerline{\date}

\tableofcontents

	\section{Introduction}
	
	Consider the family of matrices given by
	\begin{equation}\label{eq:Jn_triang_fact}	
	J_n=
	\setlength\arraycolsep{1.2pt}
	\begin{pmatrix}
	\mfrac{\alpha_1+1}{2} & \alpha_2 & \alpha_3 & \cdots & \alpha_n\\
	0   & \mfrac{\alpha_2+1}{2} & \alpha_3 & \cdots & \vdots\\
	0   &   0 & \mfrac{\alpha_3+1}{2} & \ddots & \vdots\\
	\vdots & \ddots & \ddots & \ddots & \alpha_n\\
	0 & \cdots & \cdots & 0 & \mfrac{\alpha_n+1}{2}
	\end{pmatrix}
    \setlength\arraycolsep{1.2pt}
	\begin{pmatrix}
	\mfrac{\alpha_1+1}{2} & 0 & \cdots & \cdots & 0\\
	\alpha_1 & \mfrac{\alpha_2+1}{2} & \ddots & \ddots & \vdots\\
	\vdots   &   \alpha_2 & \ddots & \ddots & \vdots\\
	\vdots   & \vdots & \ddots & \mfrac{\alpha_{n-1}+1}{2} & 0\\
	\alpha_1 & \alpha_2 & \cdots & \alpha_{n-1} & \mfrac{\alpha_n+1}{2}
	\end{pmatrix},
	\end{equation}
	where $\mathcal{A}=\{\alpha_1,\alpha_2,\ldots,\alpha_n\}\in\mathbb{R}$ are parameters. Such structured matrices arise when solving certain kinds of optimal control problems \cite{petersalg2009,petvar2017,MIDDLETON2004}, by means of the use of Nehari's Theorem \cite{chuichen12}. In particular, the eigenvalues and eigenvectors of $J_n$ are needed in order to construct an optimal controller and determine the performance achieved by it. Commonly, numerical methods are used to solve such problems, which frequently hide the behaviour of solutions in terms of the parameters defining the problem \cite{GaAp94,Tits}. In this work, we obtain a mathematical description of the eigenvalues of $J_n$ revealing several of their properties with respect to the parameter set $\mathcal{A}$. 
	
	For the aforementioned optimal control problem,  it is highly desirable to possess mathematical descriptions that allow designers to study its solutions in a deeper way. They provide tools to understand or interpret physical properties of dynamical systems under the presence of feedback. Moreover, closed form solutions for such problems, in terms of the systems parameters, reveal the best achievable performance for certain configurations of control problems, pointing engineers in the correct direction when designing or implementing less sophisticated and, most of the time, more economical control schemes. A notable example of this is when reducing the high order solutions of optimal control problems \cite{antoulas2001survey}. In this context, we believe that the results presented in this work provide sensible tools for dealing with the related optimization problems, giving insight about the nature of the eigenvalues in terms of the defining parameters. We also think that these results might be the basis for finding new algorithms and methods for related problems. 
	
	As motivation, and to highlight that the structure of $J_n$ is non-trivial, we can see that for $n=1$ the eigenvalue of $J_1$ is given by $\lambda= 0.25(\alpha_1+1)^2$. For $J_2$ we have that the two eigenvalues are 
	\begin{align*}
	\lambda=&\dfrac{(\alpha_1+1)^2+(\alpha_2+1)^2+4\alpha_1\alpha_2 \pm(\alpha_1 + \alpha_2)\sqrt{(\alpha_1+2)^2+(\alpha_2+2)^2+6\alpha_1\alpha_2-4}}{8}.
	\end{align*}
	It is of course possible to obtain algebraic expressions for $n=3$, however, it would take several lines of cumbersome expressions to write each of the three eigenvalues. In this paper we obtain results that allow to study the behaviour of the eigenvalues of $J_n$ for arbitrary values of $n$. For instance, if $\alpha_{i}\geq 0$, for all $i$,\footnote{ $\alpha_{i}\geq 0$, for all $i$ is the unique necessary condition in order to deal with  optimal control problems as the aforementioned. In this paper, however, we are not restricted to that case and results are presented for $\alpha_{i} \in\mathbb{R}$, for all $i$} we can claim that the corresponding eigenvalues $\lambda_k$, $k=1,2,\dots n$, of $J_n$  are the solutions of  
	\begin{align}\label{eq:arctan}
	\arctan\left(\dfrac{1}{\sqrt{4\lambda_k-1}}\right) +2\sum_{i=1}^{n}\arctan\left(\dfrac{\alpha_i}{\sqrt{4\lambda_k-1}}\right)=(2k-1)\dfrac{\pi}{2}
	\end{align}
	for $k=1,\ldots,n.$ Similar claims can be made whenever the eigenvalues of $J_n$ admit real solutions. Even though the properties of the $\arctan(\cdot)$ function clearly imply that \eqref{eq:arctan} should be equivalent to obtaining the roots of a polynomial, we believe that this new representation of the eigenvalue problem provides a lot of insight into the behaviour of the solutions in terms of the parameters $\alpha_i$, specially when considering the originating Control Theory problems. For example, our analysis shows that if all $\alpha_i\geq 0$, then $J_n$ has only real, positive and distinct eigenvalues. Moreover, if $\alpha_i=-\alpha_m=\alpha$ for some $i\neq m$, equation \eqref{eq:arctan} also holds, $J_n$ has an eigenvalue of multiplicity 2 at $(1-\alpha^2)/4$, while the remaining $n-2$ eigenvalues correspond to the eigenvalues of $J_{n-2}$. Also, from \eqref{eq:arctan}, it is possible to generate simple bounds for the values of the eigenvalues that do not seem obvious from the expression for $J_n$. Another benefit of the results proposed in this work correspond to the case of repeated values of the parameters. In fact, if  $\alpha_i=\alpha$ for all $i$, \eqref{eq:arctan} collapses to
	\begin{align}
	\arctan\left(\dfrac{1}{\sqrt{4\lambda_k-1}}\right) +2n\arctan\left(\dfrac{\alpha}{\sqrt{4\lambda_k-1}}\right)=(2k-1)\dfrac{\pi}{2},
	\end{align}
	and this equation provides a faster and more accurate way to compute the eigenvalues for large $n$ when compared to standard computational methods (such as the ones used by Matlab).
	
	Our study on the eigenvalues of $J_n$ was also motivated by the fact that non trivial eigenvalue problems with closed form solutions for structured matrices do exist. 
	For example, a tridiagonal Toeplitz matrix in $\mathbb{R}^{n\times n}$ has known eigenvalues in terms of the three parameters that define it \cite{yueh2005eigenvalues,noschese2013tridiagonal,BUCHHOLZER20121837}. Another well studied related problem is the inverse eigenvalue problem, which in simple words can be said equivalent to find, if possible, a matrix with fixed structure that possesses a set of eigenvalues and/or eigenvectors given a priori \cite{FRIEDLAND197715,chu1998inverse,datta2011solution,higgins2016inverse}. Our findings where sparked by a similar idea. We observed patterns when solving some eigenvalue problems which generated \eqref{eq:arctan} as a conjecture. Our goal was then to show whether it is always possible to connect \eqref{eq:arctan} to a characteristic polynomial, and eventually to our original matrices $J_n$. We notice that, in general, it is possible to write recursive relations for the characteristic polynomials of general tridiagonal matrices through the use of the \emph{Continuant} \cite{muir2003treatise,hearon1970roots}. We use this fact and other manipulations of $J_n$ in order to achieve our goal.
	
	On the other hand, it is important to note that the family of matrices $J_n$ belongs to the class of $n\times n$ rank displacement matrices where only $\mathcal{O}(n)$ parameters are needed to define them (see \cite{kailath1995displacement,pan2012structured} and the references therein). In particular, given a fixed pair of matrices $\{A,B\}$, and a field $\mathbb{F}$, a \emph{Stein} displacement operator $L$ for square matrices is defined as $L:\mathbb{F}^{n\times n}\rightarrow \mathbb{F}^{n\times n}$, such that $L=\Delta_{A,B}$, where
	\begin{align}
	L(M)=\Delta_{A,B}(M)=M-AMB.
	\end{align}
	Then, the image $L(M)$ of the operator $L$ is called the \emph{displacement} of $M$ and the rank of $L(M)$ is called the \emph{displacement rank} of $M$. Such definition aims to exploit the structure of the matrix when the displacement rank of $M$ is low and $A,B$ are simple and sparse for matrix computations, inversions, etc \cite{PANstructured,PanWang2003}. In our current case, $J_n$ is similar to the product of two matrices which are the solutions of the Lyapunov equations
	\begin{equation}
	\Delta_{A,A^\top}(M)=C,\qquad \qquad
	\Delta_{A^\top,A}(M)=K,
	\end{equation}
	with $C,K$ being of rank 1.	It is clear that the solutions of the aforementioned equations are of displacement rank 1. However, it is well known that even if two matrices are structured, their product does not necessarily inherit the same explicit structure of the factors. Nevertheless it should be expected for it to also have a low displacement rank, whenever the factors have it \cite{kailath1995displacement}. 
	
	Given the previous observations, we believe that our results complement the field of displacement structured matrices. For example, in \cite{pan2012structured,PanWang2003} the inversion of displacement operators was suggested for the solution of Nenvalinna-Pick interpolation and Nehari problems, which are common generators of ``\emph{skew-Hankel-like}" matrices, and the main motivation of our present work. It is considered in \cite{PanWang2003} that the hidden structure of matrices similar to $J_n$ is hard to exploit, however numerical methods are readily available for certain computations. Fast algorithms for computations on similar matrices can be found in \cite{bostanetall2017} and the references therein. We consider that our results provide another connection between displacement methods and optimal control problems, providing an alternative analytical description for a type of problem that seldom allows one.
	
	The presentation of our work is organized as follows. In section \ref{sec:polyJ} we obtain a recursive formula for the sequence of characteristic polynomials of $J_n$. In section \ref{sec:arctan} we derive a recursive formula for a sequence of polynomials that can be connected to \eqref{eq:arctan}. We show in Section \ref{sec:same} that both sequences have indeed the same roots. We present some applications and considerations of our findings in Section \ref{sec:numerics}. Final remarks are given in Section \ref{sec:conclusion}.
	
	\textbf{Notation:} The imaginary unit is denoted by $j^2=-1$. For a number $Q\in\mathbb{C}$, $\Re\{Q\}$ and $\Im\{Q\}$ denote its real and imaginary parts respectively and $\angle\{Q\}$ denotes its argument. For any given square matrix $M$, $\sigma(M)$ denotes its spectrum and $\rho(M)$ its spectral radius.
	
	%

	\section{Characteristic polynomial of $J_n$}\label{sec:polyJ}
	In this section we provide a recursive way to compute the characteristic polynomial of $J_n$ defined in \eqref{eq:Jn_triang_fact}. We first define the diagonal matrices
	\begin{align}\label{eq:D}
	D_\alpha=\begin{pmatrix}
	\alpha_1 &  &  & \\ 
	& \alpha_2 &  & \\
	&  & \ddots & \\
	&  &  & \alpha_n 
	\end{pmatrix},
	\hspace{3mm}
	D_\beta=\begin{pmatrix}
	\beta_1 &  &  & \\ 
	& \beta_2 &  & \\
	&  & \ddots & \\
	&  &  & \beta_n 
	\end{pmatrix}, \quad \text{with} \quad \beta_i \triangleq \dfrac{\alpha_i-1}{2\alpha_i}. 
	\end{align}
	
	We note that the particular structure of $J_n$ in \eqref{eq:Jn_triang_fact} allow us to write
	\begin{align}\label{eq:Jnfactored}
	J_n=\left[ L_1-D_{\beta}\right] D_{\alpha}\left[ L_1^\top-D_{\beta}\right] D_{\alpha},
	\end{align}
	where $L_1$ is the upper triangular matrix with $1$s at the non-zero entries. The inverse of $L_1$ and $L_1^\top$ are straightforward to compute:
	\begin{align}
	L_1^{-1}=
	\left(\begin{array}{ccccc} 
	1 & -1 & 0 & \cdots & 0\\ 
	0 & 1 & -1 & \ddots &\vdots\\
	\vdots & \ddots & 1 & \ddots & 0\\
	\vdots & \ddots & \ddots & \ddots & -1 \\
	0 & \cdots & \cdots & 0 & 1\end{array}\right),\quad (L_1^\top)^{-1}=(L_1^{-1})^\top.
	\end{align}
	We will obtain the characteristic polynomial of $J_n$ as the determinant of a tridiagonal matrix, obtaining a recursive relation for their computation. 
	\begin{proposition}\label{prop1}
		The characteristic polynomial $P_n(x)$ of the matrix $J_n$ given by \eqref{eq:Jn_triang_fact}, satisfies the recursive equation
		\begin{multline}\label{eq:recJn}
		P_n(x)=\left[ \left(\dfrac{1}{\alpha_n}+\dfrac{1}{\alpha_{n-1}}\right)x -1-\dfrac{(\alpha_n-1)^2}{4\alpha_n}-\dfrac{(\alpha_{n-1}-1)^2}{4\alpha_{n-1}}\right] P_{n-1}(x) \nonumber\\
		-\left[ \dfrac{(\alpha_{n-1}-1)^2}{4\alpha_{n-1}}-\dfrac{\alpha_{n-1}-1}{2}-\dfrac{1}{\alpha_{n-1}}x\right] ^2P_{n-2}(x),
		\end{multline}
		with initial conditions $P_0(x)=1$ and $P_1(x)=x/\alpha_1-(\alpha_1+1)^2/(4\alpha_1).$
	\end{proposition}
	\begin{proof}
		The characteristic polynomial of $J_n$ can be computed as
		\begin{align}
		P_n(x)=\det\left(xI-J_n\right),
		\end{align}
		and from \eqref{eq:Jnfactored}, multiplying accordingly, we can write
		\begin{multline}
		\det(xI-J_n)=\\
		\label{eq:det(xI-J_n)}
		\det\left(xL_1^{-1}D_{\alpha}^{-1}(L_1^{-1})^\top-D_{\alpha}+L_1^{-1}D_{\alpha}D_{\beta}
		+D_{\alpha}D_{\beta}(L_1^{-1})^\top-L_1^{-1}D_{\alpha}D_{\beta}^2(L_1^{-1})^\top\right).
		\end{multline}
		It can be checked that every term inside the determinant on the right hand side of \eqref{eq:det(xI-J_n)} is tridiagonal
		\begin{align}
		L_1^{-1}D_\alpha^{-1}(L_1^{-1})^\top=
		\begin{pmatrix}
		\dfrac{1}{\alpha_1}+\dfrac{1}{\alpha_2} & -\dfrac{1}{\alpha_2} & 0 & \cdots & 0 & 0 \\[3mm]
		-\dfrac{1}{\alpha_2} & \dfrac{1}{\alpha_2}+\dfrac{1}{\alpha_3} & -\dfrac{1}{\alpha_3} & \cdots & 0 & 0\\
		0 & -\dfrac{1}{\alpha_3} & \ddots & \cdots & 0 & 0\\
		\vdots & \vdots & \vdots & \ddots & \vdots & \vdots \\
		0 & 0 & 0 & \cdots &  \dfrac{1}{\alpha_{n-1}}+\dfrac{1}{\alpha_n} & -\dfrac{1}{\alpha_n}\\
		0 & 0 & 0 & \cdots &  -\dfrac{1}{\alpha_n} & \dfrac{1}{\alpha_n}
		\end{pmatrix}
		\end{align}
		
		\begin{multline}
		L_1^{-1}D_{\alpha}D_{\beta}^2(L_1^{-1})^\top=\\
		\begin{pmatrix}
		\alpha_1\beta_1^2+\alpha_2\beta_2^2 & -\alpha_2\beta_2^2 & 0 & \cdots & 0 & 0 \\ 
		-\alpha_2\beta_2^2 & 	\alpha_2\beta_2^2+\alpha_3\beta_3^2 & -\alpha_3\beta_3^2 & \cdots & 0 & 0\\
		0 & -\alpha_3\beta_3^2 & \ddots & \cdots & 0 & 0\\
		\vdots & \vdots & \vdots & \ddots & \vdots & \vdots \\
		0 & 0 & 0 & \cdots &  \alpha_{n-1}\beta_{n-1}^2+\alpha_n\beta_n^2 & -\alpha_n\beta_n^2\\
		0 & 0 & 0 & \cdots &  -\alpha_n\beta_n^2 & \alpha_n\beta_n^2
		\end{pmatrix}
		\end{multline}
		
		\begin{equation}
		D_{\alpha}-L_1^{-1}D_{\alpha}D_{\beta}-D_{\alpha}D_{\beta}(L_1^{-1})^\top=
		\begin{pmatrix}
		1 & \alpha_2 \beta_2 & 0 & \cdots & 0 & 0 \\ 
		\alpha_2 \beta_2 & 1 & \alpha_3 \beta_3 & \cdots & 0 & 0\\
		0 & \alpha_3 \beta_3 & \ddots & \cdots & 0 & 0\\
		\vdots & \vdots & \vdots & \ddots & \vdots & \vdots \\
		0 & 0 & 0 & \cdots &  1 & \alpha_n \beta_n\\
		0 & 0 & 0 & \cdots &  \alpha_n \beta_n & 1
		\end{pmatrix}.
		\end{equation}
		Now, we can relabel and reorder the elements inside of \eqref{eq:det(xI-J_n)} in order to compute the determinant of $\det(xI-J_n)$ as the determinant of a symmetric tridiagonal matrix
		\begin{align}
		\det(xI-J_n)=\det\begin{pmatrix}
		\gamma_1(x) & \delta_1(x) & 0 & \cdots & 0 & 0 \\ 
		\delta_1(x) & \gamma_2(x) & \delta_2(x) & \cdots & 0 & 0\\
		0 & \delta_2(x) & \ddots & \cdots & 0 & 0\\
		\vdots & \vdots & \vdots & \ddots & \vdots & \vdots \\
		0 & 0 & 0 & \cdots &  \gamma_{n-1}(x) & \delta_{n-1}(x)\\
		0 & 0 & 0 & \cdots &  \delta_{n-1}(x) & \gamma_{n}(x)
		\end{pmatrix}
		\end{align}
		where
		\begin{align}
		\gamma_1(x)=&\dfrac{x}{\alpha_1}-\dfrac{(\alpha_1+1)^2}{4\alpha_1} \\
		\gamma_n(x)=&\left(\dfrac{1}{\alpha_{n-1}}+\dfrac{1}{\alpha_n}\right)x-1-\left(\dfrac{(\alpha_{n-1}-1)^2}{4\alpha_{n-1}}+\dfrac{(\alpha_{n}-1)^2}{4\alpha_{n}}\right),\quad n>1\\
		\delta_n(x)=&\left(-x\dfrac{1}{\alpha_n}+\dfrac{(\alpha_{n}-1)^2}{4\alpha_{n}}-\dfrac{\alpha_n-1}{2}\right),\quad n\geq 1.
		\end{align}
		Using the \emph{Continuant} \cite{hearon1970roots} we can conclude that the characteristic polynomial of $J_n$ satisfies the recursive equation 
		\begin{multline}
		P_n(x)=\left[ \left(\dfrac{1}{\alpha_n}+\dfrac{1}{\alpha_{n-1}}\right)x -1-\dfrac{(\alpha_n-1)^2}{4\alpha_n}-\dfrac{(\alpha_{n-1}-1)^2}{4\alpha_{n-1}}\right] P_{n-1}(x) \\
		-\left[ \dfrac{(\alpha_{n-1}-1)^2}{4\alpha_{n-1}}-\dfrac{\alpha_{n-1}-1}{2}-\dfrac{1}{\alpha_{n-1}}x\right] ^2P_{n-2}(x)
		\end{multline}
		with initial conditions $P_0(x)=1$, $P_1(x)=x/\alpha_1-(\alpha_1+1)^2/(4\alpha_1)$. 
	\end{proof} 
	
	It is possible to verify that the recursion from Proposition \ref{prop1} generates the same polynomials as the ones obtained by directly computing $\det(xI-J_n)$ with $J_n$ defined as in \eqref{eq:Jn_triang_fact}. It can also be noted that the initial conditions $P_0(x)=\alpha_1$ and $P_1(x)=x-(\alpha_1+1)^2/(4\alpha_1)$ yield polynomials with the same roots.

	\section{A sequence of polynomials whose roots are linked to \eqref{eq:arctan}}\label{sec:arctan}
	As discussed in the introduction, for non-negative values of the parameters, we conjecture that the $n$ eigenvalues of $J_n$, $\lambda_k$, $k=1,\ldots,n$, satisfy the equations
	\begin{align}
	\arctan\left(\dfrac{1}{\sqrt{4\lambda_k-1}}\right) +2\sum_{i=1}^{n}\arctan\left(\dfrac{\alpha_i}{\sqrt{4\lambda_k-1}}\right)=(2k-1)\dfrac{\pi}{2}
	\end{align}
	for $k=1,\ldots,n$. Numerical examples suggest that these equations are indeed yielding the eigenvalues of $J_n$. We proceed to find a sequence of polynomials that have for roots the solutions of \eqref{eq:arctan}. In the first place, we note that these equations can be interpreted as the angle of a product of complex factors. A possibility is then to have a polynomial $R_n(x)$ with real coefficients, written as the sum of a polynomial with complex coefficients $Q_n(x)$ plus its conjugate $\overline{Q}_n(x)$. Therefore, some of the roots of the polynomials $R_n(x)$ (or $\Re\{Q_n(x)\}$) will be given by the values of $x$ where the complex number $Q_n(x)$ is pure imaginary, or, interchangeably, when $\angle\{Q_n(x)\}$, that is, the argument of the complex number $Q_n(x)$, is an odd multiple of $\pi/2$. We now have the following result.
	
	\begin{proposition}\label{prop2}
		Let $R_n(x)=Q_n(x)+\overline{Q}_n(x)$ be a polynomial of order $n$ with real coefficients and which has as roots the $n$ solutions of \eqref{eq:arctan} when $\alpha_i\geq 0$ for all $i$. Then, the real part of the complex polynomial $Q_n(x)$ satisfies the recurrence equation
		\begin{align}\label{eq:recQn}
		\Re\{Q_n(x)\}=&\left(\frac{(4x-1)-\alpha_n\alpha_{n-1}}{\alpha_{n-1}} \right) \left(\frac{\alpha_n +\alpha_{n-1}}{4} \right)\Re\{Q_{n-1}(x)\}\nonumber\\
		&	-\dfrac{\alpha_n}{\alpha_{n-1}} \left(x+ \frac{(\alpha_{n-1}^2 -1)}{4} \right)^2\Re\{Q_{n-2}(x)\},
		\end{align}
		with initial conditions $\Re\{Q_0(x)\}=1$ and $\Re\{Q_1(x)\}=x-(\alpha_1+1)^2/4.$
	\end{proposition}
	
	\begin{proof} 
		Firstly, we must note that there exists an infinite number of polynomials sharing roots with the solutions of \eqref{eq:arctan}, according to the earlier discussion. A general class of polynomials that shares roots with expressions similar to \eqref{eq:arctan} can be described as
		\begin{align}\label{eq:atanRn(x)}
		R_n(x)=(a+jb)\prod_{i=1}^{n}\left(v_i+jw_i\right)^2+(a-jb)\prod_{i=1}^{n}\left(v_i-jw_i\right)^2.
		\end{align}
		In this particular case we must have
		\begin{align}
		\dfrac{b}{a}=\dfrac{1}{\sqrt{4x-1}},\quad \dfrac{w_i}{v_i}=\dfrac{\alpha_i}{\sqrt{4x-1}}.
		\end{align}
		Now, we can claim that
		\begin{align}\label{eq:Qn(x)}
		R_n(x)=2\Re\{Q_n(x)\},\quad Q_n(x)=(a+jb)\prod_{i=1}^{n}\left(v_i+jw_i\right)^2.
		\end{align}
		Without loss of generality, and in order to work with monic polynomials, we consider  
		\begin{align}
		v_i=\dfrac{\sqrt{4x-1}}{2},\quad w_i=\dfrac{\alpha_i}{2},\quad a=\sqrt{4x-1},\quad b=1.
		\end{align}
		It is possible to now write
		\begin{align}
		\left(\dfrac{\sqrt{4x-1}}{2}+j\dfrac{\alpha_i}{2} \right)^2=\underbrace{\dfrac{1}{4}\left((4x-1)-\alpha_i ^2\right)}_{F_i}+j\underbrace{\dfrac{\alpha_i}{2}\sqrt{4x-1}}_{G_i},
		\end{align}
		and $Q_n(x)$ can be now written as
		\begin{align}
		Q_n(x)=(a+jb)\prod_{i=1}^{n}\left(F_i+jG_i\right)
		\end{align}
		which is a complex polynomial satisfying the following recursions
		\begin{align}\label{eq:ImQn1}
		\Re\left\{Q_n(x)\right\}=&F_n\Re\left\{Q_{n-1}(x)\right\}-G_n\Im\{Q_{n-1}(x)\}\\\label{eq:ImQn2}
		\Im\left\{Q_n(x)\right\}=&F_n\Im\left\{Q_{n-1}(x)\right\}+G_n\Re\{Q_{n-1}(x)\}.
		\end{align}
		Note how $a+jb$ disappears in these recursions. Solving for a recursion on $\Re\{Q_n(x)\}$ is not a hard task if we note that from \eqref{eq:ImQn1}
		\begin{align}
		\Im\left\{Q_{n-1}(x)\right\}=\dfrac{F_n\Re\{Q_{n-1}(x)\}-\Re\{Q_n(x)\}}{G_n},
		\end{align}
		and from \eqref{eq:ImQn2}
		\begin{align}
		\Im\left\{Q_{n-1}(x)\right\}=\dfrac{\Im\left\{Q_n(x)\right\}-G_n\Re\{Q_{n-1}(x)\}}{F_n}.
		\end{align}
		Solving for $\Im\left\{Q_n(x)\right\}$ we obtain
		\begin{align}
		\Im\left\{Q_n(x)\right\}=\left(\dfrac{F_n^2}{G_n}+G_n\right)\Re\{Q_{n-1}(x)\}-\dfrac{F_n}{G_n}\Re\{Q_n(x)\}.
		\end{align}
		If we shift this backwards by 1 and replace it in the right hand side of \eqref{eq:ImQn1} yields
		\begin{align}\label{eq:recReQn}
		\Re\{Q_n(x)\}=c_n(x) \Re\{Q_{n-1}(x)\}-d_n(x)\Re\{Q_{n-2}(x)\},
		\end{align} 
		where
		\begin{align}
		c_n(x) \triangleq & F_n+F_{n-1}\dfrac{G_n}{G_{n-1}}=\dfrac{1}{4}\left((4x-1)-\alpha_n^2\right)+ \dfrac{1}{4}\left((4x-1)-\alpha_{n-1}^2\right)\frac{\alpha_n}{\alpha_{n-1}}\nonumber\\
		& \hspace{2.6cm}=\left(\frac{(4x-1)-\alpha_n\alpha_{n-1}}{\alpha_{n-1}} \right) \left(\frac{\alpha_n +\alpha_{n-1}}{4} \right) \\
		d_n(x) \triangleq & F_{n-1}^2\dfrac{G_n}{G_{n-1}}+G_{n}G_{n-1}=\dfrac{1}{16}\left((4x-1)-\alpha_{n-1}^2\right)^2\frac{\alpha_n}{\alpha_{n-1}}+
		\dfrac{1}{4}\alpha_n\alpha_{n-1}(4x-1)\nonumber\\
		& \hspace{3.4cm}=\dfrac{\alpha_n}{\alpha_{n-1}} \left(x+ \frac{(\alpha_{n-1}^2 -1)}{4} \right)^2.
		\end{align}
		Moreover, for $n=1$, we have that 
		\begin{align}
		\dfrac{Q_1(x)}{a+jb}=\left(\dfrac{\sqrt{4x-1}}{2}+j\dfrac{\alpha_1}{2}\right)^2=\dfrac{1}{4}\left(4x-1-\alpha_1^2+j2\alpha_1\sqrt{4x-1}\right),
		\end{align}
		and $\Re\{Q_1(x)\}=0$ implies
		\begin{align}
		a\left(x-\dfrac{1+\alpha_1^2}{4}\right)-b\dfrac{\alpha_1}{2}\sqrt{4x-1}=a\left(\left(x-\dfrac{1+\alpha_1^2}{4}\right)-\dfrac{\alpha_1b}{2a}\sqrt{4x-1}\right)=0.
		\end{align}
		If $a\neq 0$ and since $b/a=1/\sqrt{4x-1}$, we have that 
		\begin{align}
		x=\dfrac{1}{4}\left(1+\alpha_1\right)^2.
		\end{align}
		Therefore, with $\Re\{Q_0(x)\}=1$ and $\Re\{Q_1(x)\}=x-(\alpha_1 +1)^2/4$, the recursion \eqref{eq:recReQn} generates polynomials with roots at the solutions of \eqref{eq:arctan}.
	\end{proof} 
	
	The previous result states that the solutions of \eqref{eq:arctan} can be associated to the roots of a sequence of polynomials. It can be checked numerically that for particular values of the parameters $\alpha_i\geq 0$, the polynomials generated by the recursion \eqref{eq:recReQn} share roots with the eigenvalues of $J_n$.
	
	\begin{remark}
		The particular case where $\alpha_i = 1$ for all $i$, corresponds to 
		\begin{align}
		J_n=\left(\begin{array}{cccc} 
		1 & 1 & \cdots & 1\\
		0 & 1 & \ddots & \vdots\\
		\vdots & \ddots & \ddots & 1\\
		0 & \cdots & 0 & 1 \end{array}\right)\left(\begin{array}{cccc} 
		1 & 0 & \cdots & 0\\
		1 & 1 & \ddots & \vdots\\
		\vdots & \ddots & \ddots & 0\\
		1 & \cdots & 1 & 1 \end{array}\right),
		\end{align}
		namely, the product of the all-ones upper and lower triangular matrices. This matrix has well known eigenvalues for any $n$ given by \cite{elliott1953characteristic}
		\begin{align}\label{eq:sigmacinfty}
		\sigma(J_n)=\left\{\dfrac{1}{4}\sec^2\left(\dfrac{\pi}{2n+1}\right),\ldots,\dfrac{1}{4}\sec^2\left(\dfrac{n\pi}{2n+1}\right)\right\}.
		\end{align}
		Moreover, in this case 
		$$F_n=\left(x-\dfrac{1}{2}\right),\quad G_n=\dfrac{\sqrt{4x-1}}{2},$$ 
		and therefore 
		$$\Re\{Q_n\}=(2x-1)\Re\{Q_{n-1}\}-x^2\Re\{Q_{n-2}\},$$ 
		with $\Re\{Q_0\}=1$ and $\Re\{Q_1\}=x-1$. This recursion generates polynomials having the same roots as \eqref{eq:sigmacinfty}.
		
		We can readily check that \eqref{eq:arctan} collapses to
		\begin{align}
		(2n+1)\arctan\left(\dfrac{1}{\sqrt{4\lambda_k-1}}\right)=(2n+1)\left(\dfrac{\pi}{2}-\arctan(\sqrt{4\lambda_k-1})\right)=\dfrac{(2k+1)\pi}{2},
		\end{align}
		from where we have that 
		\begin{align}
		\lambda_k=\dfrac{1}{4}\left(1+\dfrac{1}{\tan^2\left(\dfrac{(2k+1)\pi}{2(2n+1)}\right)}\right)=\dfrac{1}{4}\sec^2\left(\dfrac{\pi(n-k)}{2n+1}\right),
		\end{align}
		for $k=1,\ldots,n,$ which again coincides with \eqref{eq:sigmacinfty}.
	\end{remark}
	
	\begin{remark}
		The recursion \eqref{eq:recReQn} can be also associated to a tridiagonal matrix through the Generalized Continuant \cite{hearon1970roots}. Let $K_n$ be the matrix 
		\begin{align}
		\displaystyle  K_n = \begin{pmatrix} r_1 & s_2 & 0 & \ldots & 0 & 0 \\ t_2 & r_2 & s_3 & \ldots & 0 & 0 \\ 0 & t_3 & r_3 & \ldots & 0 & 0 \\ \vdots & \vdots & \vdots & \ddots & \vdots & \vdots \\ 0 & 0 & 0 & \ldots & r_{n-1} & s_{n} \\ 0 & 0 & 0 & \ldots & t_{n} & r_n \end{pmatrix}
		\end{align}
		where $r_1=\Re\{Q_1\}=x-0.25(\alpha_{1}+1)^2$,  $$r_i=\left(F_i+F_{i-1}\dfrac{G_i}{G_{i-1}}\right),$$
		for $i>1$ and $s_i,t_i$ satisfy
		$$s_it_i=\left(F_{i-1}^2\dfrac{G_i}{G_{i-1}}+G_{i}G_{i-1}\right),$$
		for $i>1.$ Then, $\det(K_n)$ generates the same sequence of polynomials provided by \eqref{eq:recReQn}. \qed
	\end{remark}
	
	\begin{remark}
		It is important to note that for Proposition \ref{prop2} to hold with respect to  \eqref{eq:arctan}, every $\alpha_i$ must be non-negative. We would like to stress that this was made in order to highlight the inspiration we had for obtaining our results. The given sequence of polynomials $\Re\{Q_n(x)\}$ that satisfies \eqref{eq:recQn} is valid for any choice of the real parameters $\alpha_i$. Moreover, the eigenvalues of $J_n$ are indeed the roots of $\Re\{Q_n(x)\}$ as we will show in the following section.
	\end{remark}

	\section{Main results}\label{sec:same}
	In this section we derive the main theorem of this work, which connects the  recursive polynomials obtained in previous sections. We also derive and discuss some of its consequences.
	\begin{theorem}\label{thm1}
		Consider the $n\times n$ matrix $J_n$ defined in \eqref{eq:Jn_triang_fact} with $\alpha_{i}\in\mathbb{R}$ for $i=1,\ldots,n$. Its $n$ eigenvalues, $\lambda_k$, $k=1,\ldots,n$, are given by the $n$ roots of the polynomials $\Re\{Q_n(x)\}$.
	\end{theorem}
	\begin{proof}
		The characteristic polynomial of $J_n$, $P_n(x)$, satisfies the recursion \eqref{eq:recJn}, namely
		\begin{align}\label{eq:PnTheo}
		P_n(x)=\gamma_n(x)P_{n-1}(x)-\delta_n(x)^2P_{n-2}(x)=\delta_{n}(x)^2\left(\dfrac{\gamma_n(x)}{\delta_{n}(x)^2}P_{n-1}(x)-P_{n-2}(x)\right) 
		\end{align}
		where
		\begin{align}
		\label{a_modificado}
		\gamma_n(x)&=\left(\dfrac{1}{\alpha_n}+\dfrac{1}{\alpha_{n-1}}\right)x -1-\dfrac{(\alpha_n-1)^2}{4\alpha_n}-\dfrac{(\alpha_{n-1}-1)^2}{4\alpha_{n-1}}\\\nonumber &=\dfrac{(\alpha_n+\alpha_{n-1})(4x-1-\alpha_n \alpha_{n-1})}{4\alpha_n \alpha_{n-1}},\\
		\label{b_modificado}
		\delta_n(x)&=\dfrac{(\alpha_{n-1}-1)^2}{4\alpha_{n-1}}-\dfrac{\alpha_{n-1}-1}{2}-\dfrac{1}{\alpha_{n-1}}x\\\nonumber
		&=-\left( \dfrac{\alpha_{n-1}^2-1+4x}{4\alpha_{n-1}}\right) ,
		\end{align}
		with initial conditions $P_0(x)=1$, and $P_1(x)=\dfrac{1}{\alpha_1}x-1-\dfrac{(\alpha_1-1)^2}{4\alpha_1}.$
		
		On the other hand, we have that the polynomials $R_n(x)$ defined in Proposition \ref{prop2} satisfy the relation
		\begin{align}
		R_n(x)=c_n(x)R_{n-1}(x)-d_n(x)R_{n-2}(x)=d_n(x)\left(\dfrac{c_n(x)}{d_n(x)}R_{n-1}(x)-R_{n-2}(x)\right),
		\end{align}
		where
		\begin{align}
		\label{c_modificado}
		c_n(x)=&\frac{\left( 4x-1-\alpha_n\alpha_{n-1}\right) \left(\alpha_n +\alpha_{n-1} \right)}{4\alpha_{n-1}}  \\
		\label{d_modificado}
		d_n(x)=&\dfrac{\alpha_n}{16\alpha_{n-1}} \left(4x-1+ \alpha_{n-1}^2  \right)^2,
		\end{align}
		with the initial conditions $R_0(x)=1$, and $R_1(x)=\left(x-\dfrac{(\alpha_1+1)^2}{4}\right)=\alpha_1 P_1(x).$
		
		We claim that $R_n(x)=\theta_nP_n(x)$ for all $n$ where $\theta_n$ is a real constant for every $n$. This would make the roots of both polynomials equal for all $n$, proving the Theorem. In order to do this, we first note from \eqref{a_modificado}, \eqref{b_modificado}, \eqref{c_modificado} and \eqref{d_modificado} that, for all $n$
		\begin{equation}
		\dfrac{c_n(x)}{d_n(x)}=\dfrac{1}{\alpha_{n-1}}\dfrac{\gamma_n(x)}{\delta_n(x)^2},\qquad \;
		d_n(x)=\alpha_n \alpha_{n-1} \delta_n(x)^2.
		\end{equation}
		In particular, for $n=2$, we have
		\begin{align}
		R_2(x)=&d_2(x)\left(\dfrac{c_2(x)}{d_2(x)}R_1(x)-R_0(x)\right).
		\end{align}
		Substituting the previous expressions and recalling that $R_0(x)=P_0(x)$ and $R_1(x)=\alpha_1 P_1(x)$ we have 
		\begin{align}
		R_2(x)=\alpha_1 \alpha_2 \delta_2(x)^2\left(\dfrac{\gamma_2(x)}{\delta_2(x)^2}P_1(x)-P_0(x)\right)
		\end{align}
		which, according to \eqref{eq:PnTheo}, yields
		\begin{align}
		R_2(x)=\alpha_1 \alpha_2 P_2(x).
		\end{align}
		We then proceed by induction. Let us assume that for $k=1,\ldots,n-1$, the following equality holds:
		\begin{align}\label{eq:induc}
		R_{k}(x)=\left(\prod_{i=1}^{{k}}\alpha_i\right)P_{k}(x),
		\end{align}
		We know that
		\begin{align}
		R_n(x)=\alpha_{n} \alpha_{n-1} \delta_n(x)^2\left(\dfrac{1}{\alpha_{n-1}}\dfrac{\gamma_n(x)}{\delta_n(x)^2}R_{n-1}(x)-R_{n-2}(x)\right).
		\end{align}
		Substituting \eqref{eq:induc} we obtain
		\begin{align}
		R_n(x)=&\alpha_{n} \alpha_{n-1} \delta_n(x)^2\left(\dfrac{1}{\alpha_{n-1}}\dfrac{\gamma_n(x)}{\delta_n(x)^2}\left(\prod_{i=1}^{{n-1}}\alpha_i\right)P_{n-1}(x)-\left(\prod_{i=1}^{{n-2}}\alpha_i\right)P_{n-2}(x)\right).
		\end{align}
		Factorizing and reordering we finally obtain
		\begin{align}
		R_n(x)=&\alpha_{n} \alpha_{n-1} \left(\prod_{i=1}^{{n-2}}\alpha_i\right)\left(\gamma_n(x)P_{n-1}(x)-\delta_n(x)^2P_{n-2}(x)\right)=\left(\prod_{i=1}^{{n}}\alpha_i\right)P_n(x),
		\end{align}
		showing than indeed $P_n(x)$ and $R_n(x)$ share the same roots.
	\end{proof}

	Theorem \ref{thm1} states that the characteristic polynomial of $J_n$ can be expressed as the sum of a complex polynomial $Q_n(x)$ and its conjugate. This complex polynomial is a simple product of $n$ quadratic factors and a first order factor. The real parameters $\alpha_i$ appear each in one of the quadratic factors. There are a few obvious observations that can be made. For example, the parameters $\alpha_i$ can be arbitrarily permuted and the eigenvalues will not change. This particular property is to be expected when considering the optimal control problems that give rise to $J_n$. In said problems, it is entirely expected to have no numerical change on the result when permuting the problem parameters. We can also note that if $\alpha_i=0$ for some $i$, then there will be an eigenvalue of multiplicity $m$ at $x=0.25$, with $m$ equal to the number of parameters equal to $0$.
	
	Some less obvious consequences are given in the following corollaries.
	\begin{corollary}\label{cor:posit}
		If the parameters $\alpha_i>0$ for all $i$, then the $n$ eigenvalues  of  $J_n$, say $\lambda_k$ for $k=1,\ldots,n$,  satisfy \eqref{eq:arctan}. Also $\lambda_k$ are real positive, distinct, and monotonically decrease with $k$.
	\end{corollary}
	\begin{proof}
		We first note that the first claim is true given the motivation behind Proposition \ref{prop2} and the results in Theorem \ref{thm1} that implies that the characteristic polynomial of $J_n$ share the same roots of the polynomial in Proposition \ref{prop2}.
		For the second part, we also note that the right hand side of \eqref{eq:arctan} increases with $k$ and that the $n$ equations correspond to the $n$ intersections of the left hand side function with the constants on the right hand side. Moreover, since $\arctan(\cdot)$ is a continuous odd function and the arguments $\alpha_i(4x-1)^{-1/2}$ are always positive, the solutions must be positive if every $\alpha_i$ is positive. It can also be noted that 
		\begin{align}
		\dfrac{d}{du}\left(\arctan\left(u\right) +2\sum_{i=1}^{n}\arctan\left(\alpha_iu\right)\right)>0,
		\end{align}
		for every $u>0$, whenever $\alpha_i>0$ for all $i$ and
		\begin{align}
		\dfrac{d}{dx}\left(\dfrac{\alpha_i}{\sqrt{4x-1}}\right)<0
		\end{align}
		for all $x>1/4$. This implies for $\alpha_i>0$ for all $i$, that the left hand side of \eqref{eq:arctan} is a strictly decreasing monotonic function that maps $x>1/4$ to the interval $(0,(2n+1)\pi/2)$  Hence, the $n$ eigenvalues $\lambda_k$ for $k=1,\ldots,n$ are positive, distinct and monotonically decrease with $k$ in this case.
	\end{proof}
	
	Corollary \eqref{cor:posit} is particularly useful to deal with the optimal control problem  that motivates the study of the eigenvalues of $J_n$ since, in that case, each $\alpha_i$ is known to be positive. The properties of the nonlinear equation \eqref{eq:arctan} allow us to analyze the solution of the optimal control problem in a more intuitive way for a deeper understanding of the underlying problems.
	
	\begin{corollary}\label{cor:repeat}
		If $\alpha=\alpha_m=-\alpha_l$ for some $m\neq l$, then $J_n$ has an eigenvalue of multiplicity $2$ at $\lambda=(1-\alpha^2)/4$ and the remaining $n-2$ eigenvalues correspond to the solutions of \eqref{eq:arctan} for $k=0,\ldots,n-3$, removing the terms for $\alpha_l$ and $\alpha_m$.
	\end{corollary}
	
	\begin{proof}
		The second part of the corollary follows directly from setting $\alpha_m=-\alpha_l$ in \eqref{eq:arctan} and from the fact that $\arctan(\cdot)$ is an odd function. This cancels out the two terms $2\arctan(\alpha_m/\sqrt{4x-1})$ and $2\arctan(\alpha_l/\sqrt{4x-1})$ and the solutions of these equations do not depend neither on $\alpha_m$ nor $\alpha_l$. We can alternatively see that the solutions of \eqref{eq:arctan} are equivalent to finding the roots of 
		\begin{align}
		R_n(x)=(\sqrt{4x-1}+j)\prod_{i=1}^{n}\left(\sqrt{4x-1}+j\alpha_i\right)^2+(\sqrt{4x-1}-j)\prod_{i=1}^{n}\left(\sqrt{4x-1}-j\alpha_i\right)^2,	
		\end{align}
		and setting $\alpha_m=\alpha=-\alpha_l$ yields
		\begin{align}
		R_n(x)=&(\sqrt{4x-1}+j)(4x-1+\alpha^2)^2\prod_{i=1,i\neq m,i\neq l}^{n}\left(\sqrt{4x-1}+j\alpha_i\right)^2+\\\nonumber
		&(\sqrt{4x-1}-j)(4x-1+\alpha^2)^2\prod_{i=1,i\neq m,i\neq l}^{n}\left(\sqrt{4x-1}-j\alpha_i\right)^2,	
		\end{align}
		and $R_n(x)=(4x-1+\alpha^2)^2\tilde{R}_{n-2}(x)$ where $\tilde{R}_{n-2}(x)$ corresponds to the polynomial $R_{n-2}(x)$ with the remaining $n-2$ parameters from the original problem relabelled.
	\end{proof}
	
	\begin{remark}
		It is interesting to note that whenever two parameters satisfy $\alpha_m=-\alpha_j$, the remaining roots do not depend on them. This invariance property seems natural considering the problems that give rise to the study of these matrices. In particular, several optimal approximation problems in Hardy spaces involve the computation of the eigenvalues of related matrices \cite{chui2012discrete}. This invariance property could be key in studying the solutions to such problems or characterizing the equivalent eigenvalue problems. A similar thing can be said about the permutation property of the parameters $\alpha_{i}$.
		
		We can also check that $J_n$ will have an eigenvalue at the origin if at least one $\alpha_i=-1$. Using Corollary \ref{cor:repeat}, $J_n$ will be nilpotent if and only if $\alpha_{i}=\pm 1$ for all $i$, and the number of parameters equal to $-1$ is equal to $\lceil n/2 \rceil$, where $\lceil \cdot \rceil$ denotes the ceiling function.
	\end{remark}
	
	\section{Examples and other miscellaneous derivations}\label{sec:numerics}
	In this section we provide some examples that illustrate applications of the results and some of their practical features.
	\subsection{Sensitivity of the eigenvalues to changes of the parameters}
	Having \eqref{eq:arctan} allows us to easily compute the partial derivatives of the solutions with respect to the parameters $\alpha_i$. Since we have that the right hand side of \eqref{eq:arctan} is a constant, the expressions are valid for every eigenvalue. In particular, using the implicit function theorem, we have that any solution $\lambda_k$ of \eqref{eq:arctan}, as a function of the parameters $\alpha_i$, satisfies
	\begin{align}\label{eq:partials}
	\dfrac{\partial \lambda_k}{\partial \alpha_i}=2\dfrac{4\lambda_k-1}{4\lambda_k-1+\alpha_i^2}\left(\dfrac{1}{2\lambda_k}+4\sum_{j=1}^{n}\dfrac{\alpha_j}{4\lambda_k-1+\alpha_j^2}\right)^{-1}.
	\end{align}
	Note that the factor $(4x-1+\alpha_i^2)$ outside of the parenthesis cancels out with a term inside it. As stated in corollary \ref{cor:posit}, for positive parameters $\alpha_i$, we have positive and distinct eigenvalues. Even though it is to be expected from the expressions for $J_n$, in this case, \eqref{eq:partials} implies that the solutions $x$ become larger with the increase of any $\alpha_i$.
	
	We can also claim, for a fixed set of $n+1$ parameters that, if $\alpha_i>0$ for all $i$,  $\rho(J_n)=\overline{\lambda}_n<\overline{\lambda}_{n+1}=\rho(J_{n+1})$. To see this consider the left hand side of \eqref{eq:arctan} for $n$ parameters and evaluate it at $\overline{\lambda}_{n+1}$
	\begin{align}
	\arctan\left(\dfrac{1}{\sqrt{4\overline{\lambda}_{n+1}-1}}\right)+2\sum_{i=1}^n\arctan&\left(\dfrac{\alpha_{i}}{\sqrt{4\overline{\lambda}_{n+1}-1}}\right)= \\\nonumber
	&\dfrac{\pi}{2}-2\arctan\left(\dfrac{\alpha_{n+1}}{\sqrt{4\overline{\lambda}_{n+1}-1}}\right).
	\end{align}
	The expression on the left above is strictly decreasing and equals $\pi/2$ only at $\overline{\lambda}_n$. Assuming that $\overline{\lambda}_n>\overline{\lambda}_{n+1}$ contradicts the latter.
	
	\subsection{Numerical computations and simple bounds}
	The alternative equation for computing the eigenvalues of $J_n$ when the parameters are non-negative is suitable for iterative methods that utilize derivatives, since these are readily available in analytical form. Moreover, Newton's method can be used to obtain easy bounds for the eigenvalues if we let
	\begin{align}
	f_{n,k}(x)=\arctan\left(\dfrac{1}{\sqrt{4x-1}}\right) +2\sum_{i=1}^{n}\arctan\left(\dfrac{\alpha_i}{\sqrt{4x-1}}\right)-(2k-1)\dfrac{\pi}{2}.
	\end{align}
	It is not hard to verify that for $n=2$ the largest eigenvalue of $J_2$ in this case satisfies
	\begin{align}
	\rho(J_2)\geq \dfrac{(\alpha_1+1)^2}{4}+\dfrac{(\alpha_2+1)^2}{4}.
	\end{align}
	Extrapolating, a simple, albeit conservative, lower bound for arbitrary $n$ is given by
	\begin{align}
	\rho(J_n)\geq x_0-\dfrac{f_{n,k}(x_0)}{f_{n,k}'(x_0)},
	\end{align}
	where 
	\begin{align}
	x_0=\sum_{i=1}^{n}\dfrac{(\alpha_i+1)^2}{4}.
	\end{align}
	
	\subsection{Properties of the roots for negative parameters}
	If any parameter $\alpha_{i}$ is negative, we cannot use \eqref{eq:arctan} in a straightforward fashion. However, $R_n(x)$ given by the recursion \eqref{eq:atanRn(x)} remains valid and all the eigenvalues of $J_n$ are roots of of $R_n(x).$ It should be clear to the reader that the roots of a polynomial are continuous functions of its parameters. So far we have shown some special cases where we can claim the exact location of some roots for particular parameter values. It is also a well known fact that in the real parameter case, roots with multiplicity come from complex conjugate pairs and split into real distinct roots as one parameter varies or they remain complex \cite{hinrichsen2005mathematical}. Hence, we have to consider the case of negative and even complex conjugate pairs as roots of $R_n(x)$ when at least one parameter is negative. We will illustrate this with the simplest case $n=1$. As we know the solution, we have that
	\begin{align}
	\lambda_1(\alpha)=\dfrac{(1+\alpha)^2}{4},
	\end{align}
	which evaluated at $Q_n(x)$ yields
	\begin{align}
	Q_n(\lambda_1(\alpha))=(\sqrt{(1+\alpha)^2-1}+j)\left(\sqrt{(1+\alpha)^2-1}+j\alpha\right)^2.
	\end{align}
	As we already know, when $\alpha\geq 0$, the solution is equivalent to solving \eqref{eq:arctan}. If $\alpha=-1$ the solution is $\lambda_1=0.$ Substituting in the previous expression we have that
	\begin{align}
	Q_n(\lambda_1(-1))=(\sqrt{-1}+j)\left(\sqrt{-1}-j\right)^2=0,
	\end{align}
	that is, both the real and imaginary part of $Q_n(x)$ are $0$ at $x=\lambda_1(-1).$ Clearly this case is not covered directly by \eqref{eq:arctan}. However, if $\alpha=-2$
	\begin{align}
	Q_n(\lambda_1(-2))=j\left(-j\right)^2=-j,
	\end{align}
	which has per argument $-\pi/2$, and is also not covered by \eqref{eq:arctan}. 
	
	We can distinguish a few extra cases for posing alternative equations. A real root that satisfies $4\lambda_k-1<0$ could be recovered by rewriting
	\begin{align}
	Q_n(x)&=(j\sqrt{1-4x}+j)\prod_{i=1}^n\left(j\sqrt{1-4x}+j\alpha_i\right)^2\\
    \nonumber	&=(-1)^nj(\sqrt{1-4x}+1)\prod_{i=1}^n\left(\sqrt{1-4x}+\alpha_i\right)^2.
	\end{align}
	Then we have (note the appearance of the superfluous solution $x=0.25$.)
	\begin{align}\label{eq:Pu}
	R_n(x)=(-1)^nj&\left((\sqrt{1-4x}+1)\prod_{i=1}^n\left(\sqrt{1-4x}+\alpha_i\right)^2+\right.\\\nonumber
	&\left.(\sqrt{1-4x}-1)\prod_{i=1}^n\left(\sqrt{1-4x}-\alpha_i\right)^2\right),
	\end{align}
	and since we are looking for real roots we have that the left term inside the parenthesis is always positive. At the same time, the right term inside the parenthesis is only negative whenever $\sqrt{1-4x}<1$. We can also note that in this case, $R_n(x)$ is always pure imaginary. However, if there exists a value of $x$ that vanishes the expression between parenthesis, we will have that $1-4x<1$, or equivalently, $x>0.$ It is possible to claim that when only one parameter $\alpha_{m}$ is negative, a necessary and sufficient condition to have an eigenvalue in $(0,0.25)$ is $\alpha_{m}\in (-1,0).$
	
	There is one exception to this. As we stated in Corollary \ref{cor:repeat}, one admitted solution for the case $\alpha_m=-\alpha_l=\alpha$ and located at $x=(1-\alpha^2)/4$, which is only negative for $|\alpha|> 1$. These solutions vanish both the real and imaginary part of $Q_n(x)$ and are the only case where real negative roots can be found as eigenvalues of $J_n$.
	
	In general, if some $\alpha_i<0$, all the real positive eigenvalues of $J_n$ satisfying $\lambda_k>0.25$ are found by solving
	\begin{align}\label{eq:atanex1}
	\arctan\left(\dfrac{1}{\sqrt{4\lambda_k-1}}\right) +2\sum_{i=1}^{n}\arctan\left(\dfrac{\alpha_i}{\sqrt{4\lambda_k-1}}\right)=(2k-1)\dfrac{\pi}{2},
	\end{align}
	where $k\in\mathbb{Z},$ that is, the right hand side of the equation above can also be a negative odd multiple of $pi/2$. The number of feasible solutions could be studied by finding the global maximum and minimum of the sum of $\arctan(\cdot)$ functions, and counting the number of times that an odd multiple of $\pi/2$ falls between this interval. Since we are assuming real parameters, the remaining eigenvalues must come in complex conjugate pairs or belong to the interval $[0,0.25]$
	
	For example, let us consider $n=5$ with $A=\{5,-0.1,3,-2,1.5\}$. In Figure \ref{fig:atanex1} we see all the solutions of \eqref{eq:atanex1} which are approximately $\lambda_1\approx 0.9821$ and $\lambda_2=22.5527$. It can be noted (note the zoom in of Figure \ref{fig:atanex1}) that since there are negative parameters, the slope of the sum of $\arctan(\cdot)$ functions is no longer negative for all $x$. It is possible, for some values of the parameters, for a single equation to have more than one solution.
	
	\begin{figure}[!th]\begin{center}
			\includegraphics[width=\textwidth]{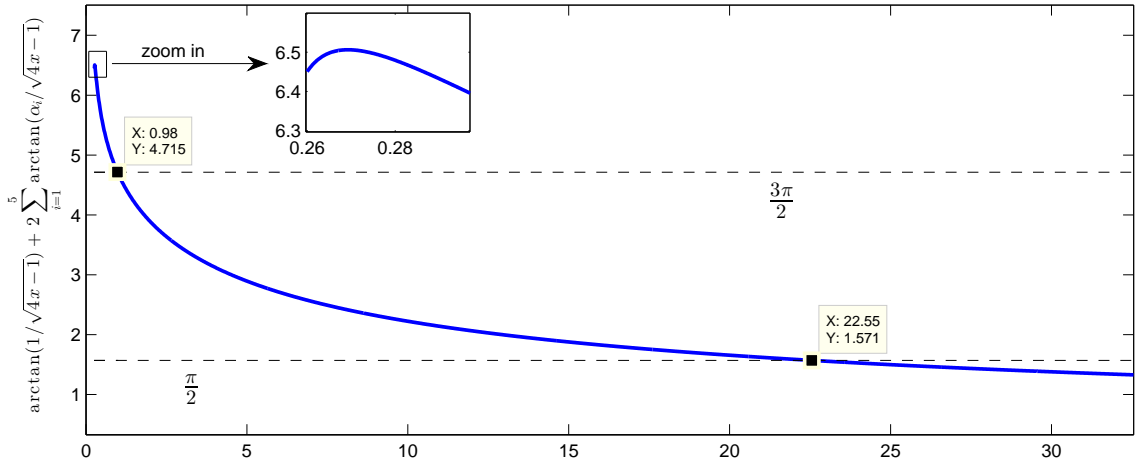}\\
			\caption{Graphical solutions of \eqref{eq:atanex1}} \label{fig:atanex1}
		\end{center}
	\end{figure} 
	
	We are still missing 3 eigenvalues. In Figure \ref{fig:Puex1} we have the plot of the imaginary part of $R_n(x)$ assuming $x<0.25$. There is a single cross with the $x$ axis at $\lambda_3\approx 0.2287<0.25$. The last two eigenvalues are a complex pair located at $\lambda_{4,5}=-0.7492\pm j0.03131$.
	
	\begin{figure}[!th]\begin{center}
			\includegraphics[width=\textwidth]{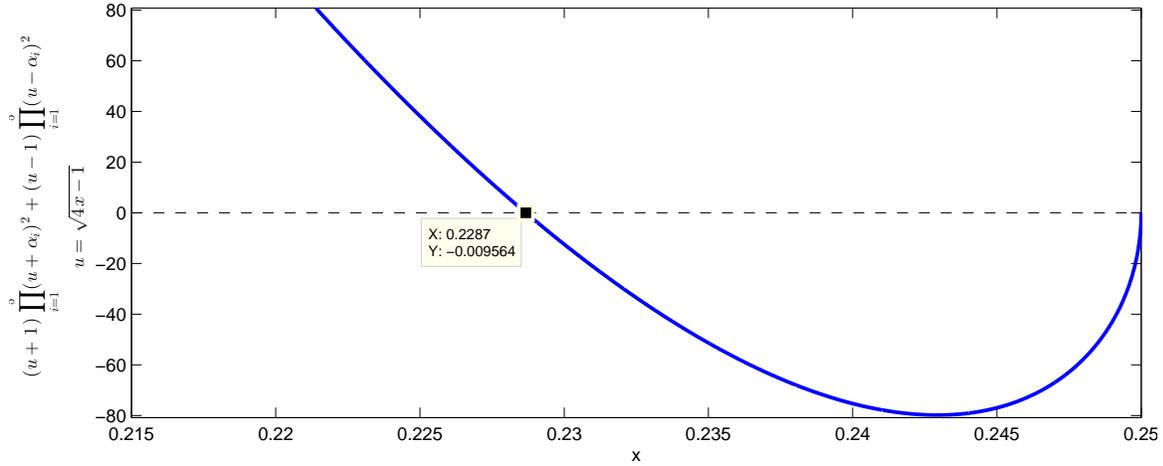}\\
			\caption{Graphical solutions of \eqref{eq:Pu}} \label{fig:Puex1}
		\end{center}
	\end{figure} 
	
	\subsection{Root locus of a particular case}
	To illustrate some of our previous comments we have plotted the root locus of $R_n(x)$ for the fixed parameters $A\setminus\{\alpha_{1}\}=\{3,-2,-5,1.5\}$ and $\alpha_{1}\in [-200,200]$. For these selections of the parameters, $J_n$ is given by
	\begin{equation}\label{eq:Jnrlocus}
	J_n=	
	\begin{pmatrix}
	\left(\dfrac{\alpha_1}{2} + \dfrac{1}{2}\right)^2 - \dfrac{5\, \alpha_1}{2} & - \dfrac{21}{2} & 8 & \dfrac{5}{2} & \dfrac{15}{8} \\[3mm]
	-\dfrac{7\, \mathrm{a1}}{2} & - \dfrac{25}{2} & 8 & \dfrac{5}{2} & \dfrac{15}{8} \\[3mm]
	- 4\, \mathrm{a1} & -12 & \dfrac{29}{4} & \dfrac{5}{2} & \dfrac{15}{8}\\[3mm]
	-\dfrac{\mathrm{a1}}{2} & - \dfrac{3}{2} & 1 & - \dfrac{7}{2} & \dfrac{15}{8}\\[3mm]
	\dfrac{5\, \mathrm{a1}}{4} & \dfrac{15}{4} & - \dfrac{5}{2} & - \dfrac{25}{4} & \dfrac{25}{16}
	\end{pmatrix}.
	\end{equation}

	We can observe that negative roots are only present when $\alpha_{1}$ equals any value in $\{-3,2,5,-1.5\}$ with locations $0.25(1-\alpha_i^2)$, that is $-6,-2,-0.75$ and $-0.3125$. There are no other real negative roots for any value of $\alpha_1$.

	\begin{figure}[!h]\begin{center}
			\includegraphics[width=\textwidth]{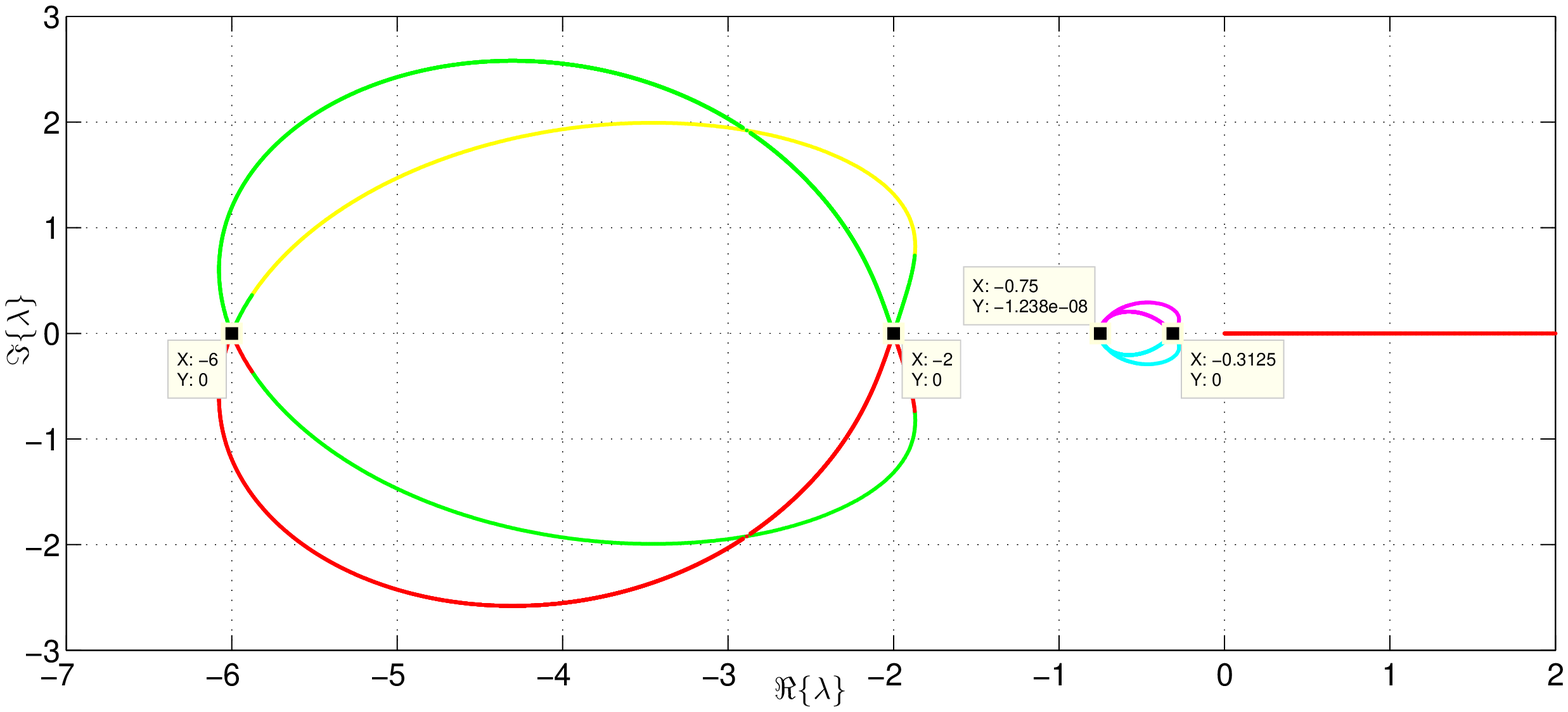}\\
			\caption{Root locus for varying $\alpha_1$} \label{fig:rootloc}
		\end{center}
	\end{figure}

	\subsection{Another closed form solution}
	
	We can use our results to conclude some other simple facts. For example, if $\alpha_1=-1$, we have 
	\begin{align}
	&R_n(x)=\\\nonumber
	&=4x(\sqrt{4x-1}-j)\prod_{i=2}^{n}\left(\sqrt{4x-1}+j\alpha_i\right)^2+4x(\sqrt{4x-1}+j)\prod_{i=2}^{n}\left(\sqrt{4x-1}-j\alpha_i\right)^2 \\\nonumber
	&=4x\left((\sqrt{4x-1}-j)\prod_{i=2}^{n}\left(\sqrt{4x-1}+j\alpha_i\right)^2+(\sqrt{4x-1}+j)\prod_{i=2}^{n}\left(\sqrt{4x-1}-j\alpha_i\right)^2\right),
	\end{align}
	and the positive roots of the expression between parenthesis are the solutions of
	\begin{align}
	-\arctan\left(\dfrac{1}{\sqrt{4\lambda_k-1}}\right)+2\sum_{i=2}^{n}\arctan\left(\dfrac{\alpha_{i}}{\sqrt{4\lambda_k-1}}\right)=(2k-1)\dfrac{\pi}{2},
	\end{align}
	with $k\in\mathbb{Z}.$ This is the same expression as the one obtained by setting $\alpha_1=-1$ in \eqref{eq:arctan}. However, if $\alpha_i=-1$ for all $i$ we have
	\begin{align}
	(1-2n)\arctan\left(\dfrac{1}{\sqrt{4\lambda_k-1}}\right)=(2k-1)\dfrac{\pi}{2},
	\end{align}
	which only has solutions when $k\leq 0$. In such case we have that the remaining $n-1$ eigenvalues in this case are
	\begin{align}
	\lambda_k=\dfrac{1}{4}\left(1+\cot^2\left(\dfrac{1+2k}{2n-1}\dfrac{\pi}{2}\right)\right)=\dfrac{1}{4}\csc^2\left(\dfrac{1+2k}{2n-1}\dfrac{\pi}{2}\right),\quad k=0,\ldots,n-2.
	\end{align}
	A simple numerical example corroborates these values. If $k=n-1$ one can note that the resulting value for the above expression is $0.25$. According to the discussion on the previous subsection, these correspond to the roots of
	\begin{align}
	R_n(x)=&4x\left((\sqrt{1-4x}-1)^{2n-1}+(\sqrt{1-4x}+1)^{2n-1}\right),
	\end{align}
	which are connected in a straightforward manner to the roots of the unity. Indeed, by setting $u=\sqrt{1-4x}$, we have that the roots in this case are given by the solutions of
	\begin{align}
	\left(\dfrac{u-1}{u+1}\right)^{2n-1}=-1.
	\end{align}

	\subsection{Limit cases}
	We finally briefly review some limit cases. If a single parameter satisfies $|\alpha_i|\rightarrow \infty$, then one eigenvalue tends to $+\infty,$ and the remaining $n-1$ eigenvalues tend to the eigenvalues of $J_{n-1}$ with $\alpha_{i}$ removed from $A$. This can be easily obtained from any of the expressions for the characteristic polynomial of $J_n$. We can see in Figure \ref{fig:rootloc} that some branches of the root locus must have a loop, when a single parameter moves from $-\infty$ to $+\infty$.

	\section{Conclusion}\label{sec:conclusion}
	We have obtained the characteristic polynomial of a structured $n\times n$ matrix with $n$ real parameters as the solution of a recursive relation. The matrix in question is fairly non trivial and for some selections of the parameters it corresponds to well studied cases. The provided framework allows to obtain a surprising amount of insight on the behaviour of the $n$ eigenvalues, in terms of the respective parameters. For example, we have shown that when they are positive, all the eigenvalues are also real and positive. Moreover, they can be computed by means of solving $n$ equations (one for each eigenvalue) involving inverse trigonometric functions with interesting properties. While the alternative method does not always allow for an exact computation of the eigenvalues, it does allow for approximations and the study of other behaviours in terms of the parameters. 
	
	A full characterization of possible negative and complex eigenvalues, when we allow for the presence of negative parameters, is missing. However, we provided some initial discussion and numerical simulations to illustrate the case. This was mostly possible due to our main results. 
	
	We believe that the class of matrices with a characteristic polynomial given as the sum of two complex conjugate polynomials, plays an important role in the study of certain optimal control problems and model order reduction techniques.

	\section*{Acknowledgement}
	The authors were supported by the Advanced Center for Electrical
	and Electronic Engineering, Basal Project FB0008, and by the Grants FONDECYT 3160738 and 1161241 CONICYT Chile.

	\bibliographystyle{abbrv}
	\bibliography{biblio}
	
\end{document}